\documentclass[12pt]{amsart}

\usepackage{amsfonts, amsthm, amsmath}

\usepackage{amscd}

\usepackage[latin2]{inputenc}

\usepackage{t1enc}

\usepackage[mathscr]{eucal}

\usepackage{indentfirst}

\usepackage{graphicx}

\usepackage{graphics}

\usepackage{pict2e}

\usepackage{mathrsfs}

\usepackage{enumerate}
\usepackage[pagebackref]{hyperref}
\hypersetup{backref, pagebackref, colorlinks=true}
\usepackage{cite}
\usepackage{color}
\usepackage{epic}
\usepackage{hyperref} 
\usepackage{ferrers}
\numberwithin{equation}{section}
\theoremstyle{plain}

\newtheorem{theorem}{Theorem}[section]

\newtheorem{lemma}[theorem]{Lemma}

\newtheorem{corollary}[theorem]{Corollary}

\newtheorem{proposition}[theorem]{Proposition}

\theoremstyle{definition}

\newtheorem{Def}[theorem]{Definition}

\newtheorem{example}[theorem]{Example}

\newtheorem{remark}[theorem]{Remark}

\newtheorem{?}[theorem]{Problem}

\def\INV{\mathrm{INV}}
\def\MAJ{\mathrm{MAJ}}
\def\MAK{\mathrm{MAK}}
\def\STAT{\mathrm{STAT}}
\def\DEN{\mathrm{DEN}}
\def\F{\mathrm{F}}
\def\IMAJ{\mathrm{IMAJ}}
\def\des{\mathrm{des}}
\def\ides{\mathrm{ides}}
\def\exc{\mathrm{exc}}
\def\Adj{\mathrm{Adj}}

\def\SS{\mathcal{S}}

\def\O{\mathcal{O}}

\def\boxit#1{\leavevmode\hbox{\vrule\vtop{\vbox{\kern.33333pt\hrule
    \kern1pt\hbox{\kern1pt\vbox{#1}\kern1pt}}\kern1pt\hrule}\vrule}}

\usepackage{collectbox}



\begin{document}

\title[STAT on words]{Mahonian STAT on rearrangement class of words}

\author[S.~Fu]{Shishuo Fu}
\address{College of Mathematics and Statistics, Chongqing University, Huxi Campus LD~506, Chongqing 401331, P.R. China.}
\email{fsshuo@cqu.edu.cn}

\author[T.~Hua]{Ting Hua}
\address{College of Mathematics and Statistics, Chongqing University, Huxi Campus, Chongqing 401331, P.R. China.}
\email{htingzl@sina.com}

\author[V.~Vajnovszki]{Vincent Vajnovszki}
\address{LE2I, Universit\'{e} de Bourgogne Franche-Comt\'{e}, BP 47870, 21078 Dijon Cedex, France}
\email{vvajnov@u-bourgogne.fr}

\date{\today}

\begin{abstract}
In 2000, Babson and Steingr\'{i}msson generalized the notion of permutation patterns to the so-called vincular patterns, and they showed that many Mahonian statistics can be expressed as sums of vincular pattern occurrence statistics. STAT is one of such Mahonian statistics discoverd by them. In 2016, Kitaev and the third author introduced a words analogue of STAT and proved a joint equidistribution result involving two sextuple statistics on the whole set of words with fixed length and alphabet. Moreover, their computer experiments hinted at a finer involution on $R(w)$,  the rearrangement class of a given word $w$. We construct such an involution in this paper, which yields a comparable joint equidistribution between two sextuple statistics over $R(w)$. Our involution builds on Burstein's involution and Foata-Sch\"{u}tzenberger's involution that utilizes the celebrated RSK algorithm. 

\end{abstract}

\keywords{permutation statistic; vincular pattern; involution; equidistribution; RSK algorithm.}

\maketitle


\section{Introduction}\label{sec:intro}
An occurrence of a classical pattern $p$ in a permutation $\pi$ is a subsequence of $\pi$ that is order-isomorphic to $p$, i.e., that has the same pairwise comparisons as $p$. For example, $41253$ has two occurrences of the pattern $3142$ in its subsequences $4153$ and $4253$. The study of (classical) permutation patterns have mostly been with respect to \emph{avoidance}, i.e., on the enumeration of permutations in the symmetric group $\SS_n$ that have no occurrences of the pattern(s) in question. The tone is generally believed to be set in 1968, when Knuth published Volume one of \emph{The Art of Computer Programming} \cite{Knu}. It was observed in \cite{Knu} that the number of $231$-avoiding permutations are enumerated by the Catalan numbers. See \cite{Kit} and the references therein for a comprehensive introduction to patterns in permutations.

In 2000, Babson and Steingr\'{i}msson \cite{BS} traced out a new line of research by generalizing classical patterns to what are now known as \emph{vincular patterns}. In an occurrence of a vincular pattern, some letters of that order-isomorphic subsequence may be required to be adjacent in the permutation, and this is done by underscoring the adjacent letters in the expression for the pattern. For our previous example, only $4153$ represents an occurrence of the vincular pattern $\underline{31}42$ in $41253$. As was investigated thoroughly in \cite{BS}, these vincular patterns turned out to be the ``building blocks'' in their classification of Mahonian statistics. Recall that a permutation statistic is called \emph{Mahonian} precisely when it has the same distribution, on $\SS_n$, as $\INV$, the number of inversions. Given a vincular pattern $\tau$ and a permutation $\pi$, we denote by $(\tau)\pi$ the number of occurrences of the pattern $\tau$ in $\pi$. When it causes no confusion, we will even drop the parentheses. For instance, it was observed/defined in \cite{BS}:
\begin{align}
\INV &=21=\underline{21}+3\underline{12}+3\underline{21}+2\underline{31},\nonumber\\
\MAJ &=\underline{21}+1\underline{32}+2\underline{31}+3\underline{21},\nonumber\\
\MAK &=\underline{21}+1\underline{32}+2\underline{31}+\underline{32}1,\nonumber\\
\STAT &=\underline{21}+\underline{13}2+\underline{21}3+\underline{32}1.
\end{align}
Note that the first three Mahonian statistics are well-known ones putting on a new look, while the last one $\STAT$ was one of those introduced in \cite{BS} and shown to be Mahonian.

Now we state the joint equidistribution results involving STAT obtained by Burstein \cite{Bur} on $\SS_n$ and by Kitaev and the third author \cite{KV} on $[m]^n$, the set of length $n$ words over the alphabet $[m]:=\{1,2,\ldots,m\}$. These two results largely motivated this work. All undefined statistics will be introduced in the next section. For the spelling of the statistics, we follow the convention of \cite{CSZ}, i.e., all Mahonian statistics are spelled with uppercase letters, all Eulerian statistics are spelled with lowercase letters, while all the remaining ones are merely capitalized. Moreover, the italic letters are used for set-valued statistics, such as {\it Id} in Theorem~\ref{main-Sn} and Corollary~\ref{main-Rw}.
\begin{theorem}[Theorem~2.1 in \cite{Bur}]\label{thm-Bur}$\ $\\
Statistics $(\Adj, \des, \F, \MAJ, \STAT)$ and
$(\Adj, \des, \F, \STAT, \MAJ)$ have the same joint distribution on $\SS_n$ for all $n$.
\end{theorem}
\begin{theorem}[Theorem~2 in \cite{KV}]\label{The2}$\ $\\
Statistics $(\Adj, \des, \ides, \F, \MAJ, \STAT)$ and
$(\Adj, \des, \ides, \F, \STAT, \MAJ)$ have the same joint distribution on $[m]^n$ for all $m$ and $n$.
\end{theorem}

A quick look at Table~\ref{1122} reveals that when one wants to extend the above results to the rearrangement class $R(w)$, the statistic $\Adj$ must be dropped. In the meantime, the involution we are going to construct actually preserves the inverse descent set {\it Id} (see Definition~\ref{shuffle}). The following are the main results of this paper.

\begin{theorem}\label{main-Sn}$\ $\\
Statistics $(\des, Id, \F, \MAJ, \STAT)$ and
$(\des, Id, \F, \STAT, \MAJ)$ have the same joint distribution on $\SS_n$ for all $n$.
\end{theorem}
\begin{remark}
It is worth mentioning that in \cite{FS} it has been shown that $(Id, \MAJ, \INV)$ and
$(Id, \INV, \MAJ)$ are equidistributed on $\SS_n$ for all $n$.
\end{remark}

In view of this theorem and the standard ``coding/decoding'' between words and permutations \cite{CSZ}, we derive the following version for words, where the statistics have to be redefined on words accordingly. We defer the details to the next section.
\begin{corollary}\label{main-Rw}$\ $\\
Statistics $(\des, Id, \F, \MAJ, \STAT)$ and
$(\des, Id, \F, \STAT, \MAJ)$ have the same joint distribution on $R(w)$ for any word $w$.
\end{corollary}

Since $Id$ uniquely determinates both $\ides$ and $\IMAJ$ (see Definition~\ref{shuffle}), the next corollary is a consequence of the previous one, and it is similar to Theorem~\ref{The2}, except
$\Adj$ is replaced by $\IMAJ$ and the equidistribution
is stated on the rearrangement classes.
\begin{corollary}$\ $\\
Statistics $(\IMAJ, \des, \ides, \F, \MAJ, \STAT)$ and
$(\IMAJ, \des, \ides, \F, \STAT, \MAJ)$ have the same joint distribution on $R(w)$ for any word $w$.
\end{corollary}

In the next section, we will first present the formal definitions of all the statistics that concern us here, and then explain the transition from Theorem~\ref{main-Sn} to Corollary~\ref{main-Rw}. Theorem~\ref{main-Sn} will be proved in Section~\ref{sec:involution}, where Foata-Sch\"{u}tzenberger's involution via RSK correspondence will be recalled. We conclude with some further questions suggested by this work.

\begin{table}[htbp]
\centering \caption{The joint distribution of $7$-tuple statistics on $R(1122)$}
\begin{tabular}{|c||c|c|c|c|c|c|c|}
\hline
$R(1122)$ & $\Adj$ & $\des$ & $\ides$ & $\F$ & $\IMAJ$ & $\MAJ$ & $\STAT$\\
\hline
1122 & 0 & 0 & 0 & 1 & 0 & 0 & 0\\
\hline
1212 & 1 & 1 & 1 & 1 & 2 & 2 & 3\\
\hline
1221 & 0 & 1 & 1 & 1 & 2 & 3 & 2\\
\hline
2112 & 0 & 1 & 1 & 2 & 2 & 1 & 2\\
\hline
2121 & 0 & 2 & 1 & 2 & 2 & 4 & 4\\
\hline
2211 & 0 & 1 & 1 & 2 & 2 & 2 & 1\\
\hline
\end{tabular}
\label{1122}
\end{table}

\section{Preliminaries}\label{sec:pre}
Since all the statistics that concern us here can be defined on both permutations and words, and that we have two versions of joint equidistribution results, some care needs to be taken to avoid unnecessary repetitions. We decide to define all the statistics on the rearrangement class $R(w)$ for a word $w$ of length $n$, making sure that when $w$ is composed of distinct letters, and thus $R(w)$ becomes essentially $\SS_n$, all the statistics reduce to the original ones defined on $\SS_n$. On the other hand, for the two equidistribution results we first explain how to derive Corollary~\ref{main-Rw} from Theorem~\ref{main-Sn} and then construct, in the next section, the key involution on $\SS_n$.

We consider words $w=w_1w_2\cdots w_n$ on a totally ordered alphabet $\mathscr{A}$. Without loss of generality, we can always take $\mathscr{A}$ to be the interval $[m]=\{1,2,\ldots,m\}$. By \emph{rearrangement class} $R(w)$, we mean the set of all words that can be obtained by permuting the letters of $w$. Within each class $R(w)$, there is a unique non-decreasing word that we denote by $\overline{w}$. If the letters of $w$ are distinct, then $R(w)$ is in bijection with the symmetric group $\SS_n$ via the following ``coding'' map \cite{CSZ}.
\begin{Def}\label{c}
Given a word $w=w_1w_2\cdots w_n$, we \emph{code} it by the permutation $c(w)=c(w)_1c(w)_2\cdots c(w)_n\in\SS_n$ uniquely defined by
\begin{align*}
c(w)_i < c(w)_j & \text{ if and only if } \begin{cases} w_i<w_j, & \text{or}\\ w_i=w_j, & \text{for }i<j.\end{cases}
\end{align*}
This gives rise to a map $c:[m]^n\rightarrow \SS_n$, which we call \emph{the coding map}, and clearly $c$ is a 
surjection if and only if $m\geq n$.
\end{Def}
For example, the word $w=2~1~2~2~3~1$ is coded by $c(w)=3~1~4~5~6~2$. Now we are ready to define all seven statistics encountered in the introduction. The true or false function $\chi$ is used in the definition of $\Adj$ for our 
convenience\textcolor{red}{:} $\chi(S)=1$ (resp.~$\chi(S)=0$) if the statement $S$ is true (resp. false).

\begin{Def}\label{shuffle}
Let $w=w_1w_2\cdots w_n$ be a word. The \emph{descent set of $w$}, $D(w)$, the \emph{inverse descent set of $w$}, $Id(w)$, and a peculiar index set (needed in Section~\ref{sec:involution}) that we call the \emph{shuffle set of $w$}, $Sh(w)$ are set-valued statistics defined as follows:
\begin{align*}
D(w)&=\{i: w_i>w_{i+1}, 1\leq i<n\},\\
Id(w)&=\{c(w)_i: \exists~j<i, \text{ s.t. } c(w)_j=c(w)_i+1\},\\
Sh(w)&=\{i: w_i\geq w_1>w_{i+1} \text{ or }w_i< w_1\leq w_{i+1}, 1\leq i<n\}.
\end{align*}
We define the following seven statistics on $w$:
\begin{align*}
\F~w &= w_1,\quad \des~w = \sum_{j\in D(w)}1,\quad
\ides~w = \sum_{j\in Id(w)}1,\\
\Adj~w &= \sum_{j=1}^{n}\chi(c(w)_j=c(w)_{j+1}+1), \text{ take }c(w)_{n+1}=0,\\
\MAJ~w &=\sum_{j\in D(w)}j=(1\underline{32}+1\underline{21}+2\underline{31}+2\underline{21}+3\underline{21}+\underline{21})w,\\
\IMAJ~w &= \sum_{j\in Id(w)}j,\\
\STAT~w &= (\underline{21}3+\underline{21}2+\underline{13}2+\underline{12}1+\underline{32}1+\underline{21})w.
\end{align*}
\end{Def}
We note that $\ides, \Adj, \MAJ, \STAT$ all have been extended to words in \cite{KV}. Some of the definitions given here may look different but essentially are equivalent to those used in \cite{KV}, and our formulation of $\ides$ makes it more convenient to define $\IMAJ$ on words.

When the word $w$ happens to be a permutation in $\SS_n$, then $c(w)=w$ and all the above statistics agree with their original counterparts defined on permutations. This hopefully justifies our abuse of notation, i.e., we use the same name for these statistics, both on words and permutations. Moreover, it should be routine to verify that the coding map $c$ preserves six statistics defined above, with $\F$ being the only exception.
\begin{align}\label{6w=6cw}
(\Adj, \des, Id, \MAJ, \STAT)w &=(\Adj, \des, Id, \MAJ, \STAT)c(w).
\end{align}
For the previous example with $w=2~1~2~2~3~1$ and $c(w)=3~1~4~5~6~2$, we see $\des~w=2=\des~c(w)$, $Id~w=\{2\}=Id~c(w)$, $\MAJ~w=6=\MAJ~c(w)$, etc.

In order to deduce Corollary~\ref{main-Rw} from Theorem~\ref{main-Sn}, we need to pay some attention to the ``decoding'' map from $\SS_n$ back to $[m]^n$. In general, the preimage of a given permutation under the coding map $c$ is not unique and therefore this decoding is not well defined. For instance, for both $w=2~1~2~2~3~1$ and $v=2~1~2~3~3~1$, we see $c(w)=c(v)=3~1~4~5~6~2$, with $R(w)\neq R(v)$. This should be somewhat expected since when $w$ contains repeated letters, we have $|R(w)|<|\SS_n|$. To deal with this, we introduce the notion of compatible permutations.
\begin{Def}
Given the rearrangement class $R(w)$ for a word $w$ of length $n$, we denote $\SS_{R(w)}$ the subset of $\SS_n$ that is the image of $R(w)$ under the coding map $c$, and we call it the set of \emph{compatible permutations with respect to $R(w)$}. More precisely,
$$
\SS_{R(w)}=\{\pi\in\SS_n: \pi=c(v), \text{ for some }v\in R(w)\}.
$$
\end{Def}
It is clear from Definition~\ref{c} that if we restrict $c$ on a rearrangement class $R(w)$, it is injective, hence it induces a bijection between $R(w)$ and $\SS_{R(w)}$, with a unique inverse $c^{-1}$ defined on $\SS_{R(w)}$:
\begin{align}
c=c|_{R(w)}:& \quad R(w)\rightarrow \SS_{R(w)},\nonumber\\
c^{-1}:& \quad \SS_{R(w)}\rightarrow R(w),\nonumber\\
\label{dc=1}c^{-1}\circ c(v)&=v, \quad\forall v\in R(w),\\
\label{cd=1}c\circ c^{-1}(\pi)&=\pi, \quad\forall \pi\in \SS_{R(w)}.
\end{align}

The next proposition characterizes $\SS_{R(w)}$ using $\overline{w}$ and the inverse descent set.

\begin{proposition}\label{charSRw}
Given $R(w)$ and suppose $\overline{w}=u_1\cdots u_{b_1}u_{b_1+1}\cdots u_{b_2}\cdots u_{b_k+1}\cdots u_{n}$, where $1\leq b_1<b_2<\cdots<b_k\leq n$ and
$$
u_1=\cdots=u_{b_1}<u_{b_1+1}=\cdots=u_{b_2}<\cdots<u_{b_k+1}=\cdots=u_n.
$$
Then $\pi\in\SS_{R(w)}$ if and only if $Id(\pi)\subseteq \{b_1, b_2,\ldots, b_k\}$.
\end{proposition}
\begin{proof}
Take a permutation $\pi\in\SS_{R(w)}$, then $\pi=c(v)$ for some $v\in R(w)$, and consequently $Id(\pi)\subseteq \{b_1, b_2,\ldots, b_k\}$, this takes care of the ``only if'' part. We finish the proof by simply noting that both $\SS_{R(w)}$ and $\{\pi\in\SS_n:Id(\pi)\subseteq \{b_1, b_2,\ldots, b_k\}\}$ are enumerated by $|R(w)|=\binom{n}{b_1,b_2-b_1,\ldots,n-b_k}$.
\end{proof}


In the next section, we are going to construct an involution $\phi$ on $\SS_n$ such that for any $\pi\in\SS_n$,
\begin{align}\label{phistats}
(\des, Id, \F, \MAJ, \STAT)\pi &=(\des, Id, \F, \STAT, \MAJ)\phi(\pi).
\end{align}
This proves Theorem~\ref{main-Sn}. Now consider the function composition
\begin{align*}
\phi_{R(w)}&=c^{-1}\circ \phi \circ c.
\end{align*}
Note that for any $v\in R(w)$, $c(v)\in \SS_{R(w)}$, and since $Id(\phi(c(v)))=Id(c(v))$, we have $\phi(c(v))\in\SS_{R(w)}$ thanks to Proposition~\ref{charSRw}, thus $\phi_{R(w)}: R(w)\rightarrow R(w)$ is well-defined. Now we show that $\phi_{R(w)}$ is an involution on $R(w)$ that proves Corollary~\ref{main-Rw}.
\begin{proof}[Proof of Corollary~\ref{main-Rw}]
We first prove that $\phi_{R(w)}$ is indeed an involution. Let $v\in R(w)$, we have
\begin{align*}
\phi_{R(w)}^2(v)&=(c^{-1}\circ \phi \circ c)\circ(c^{-1}\circ \phi \circ c)v\\
&=(c^{-1}\circ \phi)\circ(c\circ c^{-1})\circ(\phi \circ c)v\\
\text{\eqref{cd=1}}\qquad\qquad\quad &=(c^{-1}\circ \phi)\circ(\phi \circ c)v\\
&=(c^{-1}\circ(\phi\circ\phi)\circ c)v\\
\phi \text{ is an involution }\quad\quad &=(c^{-1}\circ c)v\\
\text{\eqref{dc=1}} \qquad \qquad \quad &=v.
\end{align*}
Next we apply \eqref{6w=6cw} (twice) and \eqref{phistats} to get
\begin{align*}
(\des, Id, \MAJ, \STAT)v &=(\des, Id, \STAT, \MAJ)\phi_{R(w)}(v).
\end{align*}
Finally it is straightforward to check that $\F~v=\F~\phi_{R(w)}(v)$ since $\F~\pi=\F~\phi(\pi)$ by relation \eqref{phistats}. This immediately establishes Corollary~\ref{main-Rw} and ends the proof.
\end{proof}

\section{Involution $\phi$}\label{sec:involution}
In this section we aim to construct the aforementioned map $\phi$ from $\SS_n$ into itself, and we show that i) it is an involution and ii) it satisfies \eqref{phistats}. First off, we recall an involution $\jmath$ defined on $\SS_n$ that was first considered by Sch\"{u}tzenberger \cite{Schu} and further explored by Foata-Sch\"{u}tzenberger \cite{FS}. It builds on two permutation symmetries, namely \emph{reversal} r and \emph{complement} c, as well as the famous Robinson-Schensted-Knuth (RSK) algorithm \cite[$\S$7.11]{Stan}. We collect the relevant definitions and notations below. Given a permutation $\pi=\pi_1\pi_2\cdots \pi_n\in\SS_n$,
\begin{align*}
&\pi^r=\pi_1^r\pi_2^r\cdots\pi_n^r, \text{ with } \pi^{r}_i=\pi_{n+1-i}, \text{ for }1\leq i\leq n,\\
&\pi^c=\pi_1^c\pi_2^c\cdots\pi_n^c, \text{ with } \pi^{c}_i=n+1-\pi_i, \text{ for }1\leq i\leq n,\\
&\pi^{rc}=\pi_1^{rc}\pi_2^{rc}\cdots\pi_n^{rc}, \text{ with } \pi^{rc}_i=n+1-\pi_{n+1-i}, \text{ for }1\leq i\leq n,\\
&\pi\stackrel{RSK}{\longrightarrow} (P_{\pi},Q_{\pi}) \stackrel{RSK^{-1}}{\longrightarrow} \pi,
\end{align*}
where $RSK$ is a bijection between $\SS_n$ and the set of ordered pairs of $n$-cell standard Young tableaux (SYT) with the same shape. Note that $\pi\mapsto \pi^r$, $\pi\mapsto \pi^c$ and $\pi\mapsto \pi^{rc}=\pi^{cr}$ are involutions on $\SS_n$.

We usually call $P_{\pi}$ the \emph{insertion tableau} and $Q_{\pi}$ the \emph{recording tableau} of $\pi$. In general, for two unrelated permutations $\pi,\sigma\in\SS_n$, $P_{\pi}$ and $Q_{\sigma}$ are unlikely to be of the same shape. However, it was noted in \cite[Prop.~5.2]{FS} that $P_{\pi}$ and $Q_{\pi^{rc}}$ are indeed of the same shape, so the following map $\jmath$ is a well-defined involution on $\SS_n$.
\begin{align*}
\jmath:\;& \SS_n\rightarrow\SS_n\\
& \pi\mapsto RSK^{-1}(P_{\pi},Q_{\pi^{rc}})
\end{align*}
The next property of $\jmath$ is of great importance for our later use and is the main reason that urged us to involve $\jmath$ in our construction of $\phi$.
\begin{lemma}[Theorem~2 in \cite{FS}]
For any $\pi\in\SS_n$, the map $\jmath$ as defined above preserves the inverse descent set $Id(\pi)$ and exchanges the descent set $D(\pi)$ with its complement to $n$. In other words one has simultaneously:
\begin{align}
\label{jId}Id(\jmath(\pi))&=Id(\pi), \text{ and}\\
\label{jD}D(\jmath(\pi))&=\{n-k: k\in D(\pi)\}.
\end{align}
\end{lemma}


Now given a permutation $\pi\in\SS_n$ with $\F~\pi=k$, let us denote as $\pi^t=\pi_{t_1}\cdots\pi_{t_{n-k}}$ (resp. $\pi^b=\pi_{b_1}\cdots\pi_{b_{k-1}}$) the \emph{top} (resp. \emph{bottom}) subword that is composed of all the letters larger (resp. smaller) than $k$. We also need the shuffle set $Sh(\pi)$ defined in Definition~\ref{shuffle}. For example, let $\pi=5~4~6~7~3~1~9~8~2$, then $\pi^t=6~7~9~8$, $\pi^b=4~3~1~2$ and $Sh(\pi)=\{1,2,4,6,8\}$. Conversely, it should be clear how to recover the unique permutation $\pi$ that corresponds to an appropriate triple $(\pi^t,\pi^b,Sh(\pi))$, so we omit the details. We are ready to describe our key map $\phi$ using this defining triple: (top subword, bottom subword, shuffle set).
\begin{Def}
Given $\pi\in\SS_n$, let $\phi(\pi)$ be the unique permutation that corresponds to $(\phi(\pi)^t,\phi(\pi)^b,Sh(\phi(\pi)))$, where
\begin{align}
\label{phit}\phi(\pi)^t &:= c^{-1}\circ \jmath \circ c(\pi^t),\\
\label{phib}\phi(\pi)^b &:= \jmath(\pi^b),\\
\label{phish}Sh(\phi(\pi)) &:= \begin{cases}\{1\}\cup\{n+1-k:k\geq 2, k\in Sh(\pi)\} & \text{ if }|Sh(\pi)|\text{ is odd,}\\
\{n+1-k:k\geq 2, k\in Sh(\pi)\} & \text{ if }|Sh(\pi)|\text{ is even.}
 \end{cases}
\end{align}
\end{Def}
\begin{remark}
It should be noted that the idea to associate $\pi$ with the triple $(\pi^t,\pi^b,Sh(\pi))$ is motivated by Burstein's involution $p$ introduced in \cite{Bur} to prove Theorem~\ref{thm-Bur}. Actually in our setting, Burstern's $p$ can be defined using the same \eqref{phish}, together with \eqref{phit} and \eqref{phib} replaced by
\begin{align*}
p(\pi)^t &= c^{-1}((c(\pi^t))^{rc}), \text{ and }\\
p(\pi)^b &= (\pi^b)^{rc}.
\end{align*}
Moreover, in view of \eqref{jId} and Proposition~\ref{charSRw}, \eqref{phit} is well-defined. Thirdly, by the above definition, it follows that $\F~\phi(\pi)=\F~\pi$. Lastly, using the fact that $\jmath$ is an involution, it is routine (using the associated triple) to check that $\phi^2(\pi)=\pi$ for any $\pi\in\SS_n$, i.e., $\phi$ is an involution on $\SS_n$ as claimed.
\end{remark}

All remained to be done is to prove \eqref{phistats}. Namely, we need to show for every $\pi\in\SS_n$:
\begin{align*}
\des~\phi(\pi)&=\des~\pi,\\
Id(\phi(\pi))&=Id(\pi),\\
\F~\phi(\pi)&=\F~\pi,\\
\STAT~\phi(\pi)&=\MAJ~\pi,\\
\MAJ~\phi(\pi)&=\STAT~\pi.
\end{align*}
We break it up into the following three lemmas and one corollary.
\begin{lemma}
$Id(\phi(\pi))=Id(\pi).$
\end{lemma}
\begin{proof}
First note that when the first letter $\F~\pi$ is fixed, $Id(\pi)$ only depends on $Id(\pi^t)$ and $Id(\pi^b)$, and is not affected by $Sh(\pi)$ at all. Then we simply use \eqref{jId} to finish the proof.
\end{proof}

\begin{lemma}\label{M+S}
$\MAJ~\pi+\STAT~\pi=(n+1)\des~\pi-(\F~\pi-1)$.
\end{lemma}
\begin{proof}
This is precisely Lemma~2.7 in \cite{Bur}, hence we omit the proof.
\end{proof}

\begin{lemma}\label{M+M}
$\MAJ~\pi+\MAJ~\phi(\pi)=(n+1)\des~\pi-(\F~\pi-1)$.
\end{lemma}
\begin{proof}
The basic idea is to pair the descents of $\pi$ with the descents of $\phi(\pi)$. We discuss by two cases.
\begin{itemize}
\item If $i\in D(\pi)\backslash Sh(\pi)$, then either $\pi_i>\pi_{i+1}>\pi_1$ or $\pi_1>\pi_i>\pi_{i+1}$. Using \eqref{jD} and \eqref{phish} we see that either $\phi(\pi)_{n+1-i}>\phi(\pi)_{n+2-i}>\phi(\pi)_1$ or $\phi(\pi)_1>\phi(\pi)_{n+1-i}>\phi(\pi)_{n+2-i}$, both of which lead to $(n+1-i)\in D(\phi(\pi))\backslash Sh(\phi(\pi))$. So this pair contributes $$i+(n+1-i)=n+1$$ to the sum $\MAJ~\pi+\MAJ~\phi(\pi)$.
\item If $i\in D(\pi)\cap Sh(\pi)$, then $\pi_i\geq \pi_1>\pi_{i+1}$, and according to \eqref{phish}, $\phi(\pi)_{n+1-i-s_i}\geq \phi(\pi)_1>\phi(\pi)_{n+2-i-s_i}$, where $s_i$ is the length of the maximal block of consecutive entries of $\phi(\pi)$ that ends with $\phi(\pi)_{n+1-i}$ and is contained in $\phi(\pi)^b$. Therefore this pair contributes $$i+(n+1-i-s_i)=n+1-s_i$$ to the sum $\MAJ~\pi+\MAJ~\phi(\pi)$.
\end{itemize}
All such $s_i$ as defined above clearly sum up to the total length of $\phi(\pi)^b$:
$$
\sum_{i\in D(\pi)\cap Sh(\pi)}s_i=|\phi(\pi)^b|=|\pi^b|=\F~\pi-1.
$$
We add the contributions from both cases to arrive at the desired identity.
\end{proof}

\begin{corollary}
$\STAT~\phi(\pi)=\MAJ~\pi,\;
\MAJ~\phi(\pi)=\STAT~\pi, \;\des~\phi(\pi)=\des~\pi.$
\end{corollary}
\begin{proof}
It immediately follows from Lemma~\ref{M+S} and \ref{M+M} that $\MAJ~\phi(\pi)=\STAT~\pi$, and since $\phi$ is an involution we also get $\STAT~\phi(\pi)=\MAJ~\pi$. Finally we have:
\begin{align*}
(n+1)\des~\phi(\pi)-(\F~\phi(\pi)-1) &\stackrel{\text{\tiny Lemma~\ref{M+M}}}{=} \MAJ~\phi(\pi)+\MAJ~\phi^2(\pi)\\
&\stackrel{\phantom{aaa}\phi^2=1\phantom{aa}}{=} \STAT~\pi+\MAJ~\pi\\
&\stackrel{\text{\tiny Lemma~\ref{M+S}}}{=} (n+1)\des~\pi-(\F~\pi-1).
\end{align*}
This directly leads to $\des~\phi(\pi)=\des~\pi$ since we already noted that $\F~\phi(\pi)=\F~\pi$.
\end{proof}
This involution $\phi$ is admittedly a bit involved, so we include one complete example here to illustrate it.
\begin{example}
Given $w=1~1~2~3~4~4~4~5~6$, take $v=4~3~4~4~2~1~6~5~1 \in R(w)$. We first compute the statistics on $v$:
$$
(\des, Id, \F, \MAJ, \STAT)v=(5, \{2,3,4,8\}, 4, 25, 21).
$$
Now let $\pi=c(v)=5~4~6~7~3~1~9~8~2\in \SS_9$, we proceed to find $\phi(\pi)$. First we get the defining triple associated with $\pi$:
$$
(\pi^t,\pi^b,Sh(\pi))=(6~7~9~8,\; 4~3~1~2,\; \{1,2,4,6,8\}).
$$
Applying $\phi$ and according to \eqref{phish}, we get immediately $Sh(\phi(\pi))=\{1,2,4,6,8\}$. But for $\phi(\pi)^t$ and $\phi(\pi)^b$, which involve the map $\jmath$, we need to invoke the RSK algorithm. We first apply RSK on $c(\pi^t)=1~2~4~3$ and $c(\pi^t)^{rc}=2~1~3~4$ to get the following four SYT.
\begin{ferrers}
\addcellrows {3+1}
\highlightcellbyletter{1}{1}{1}
\highlightcellbyletter{1}{2}{2}
\highlightcellbyletter{1}{3}{3}
\highlightcellbyletter{2}{1}{4}
\addtext{0.8}{-1.6}{$P_{c(\pi^t)}$}
\putright
\addcellrows {3+1}
\highlightcellbyletter{1}{1}{1}
\highlightcellbyletter{1}{2}{2}
\highlightcellbyletter{1}{3}{3}
\highlightcellbyletter{2}{1}{4}
\addtext{0.8}{-1.6}{$Q_{c(\pi^t)}$}
\putright
\addcellrows {3+1}
\highlightcellbyletter{1}{1}{1}
\highlightcellbyletter{1}{2}{3}
\highlightcellbyletter{1}{3}{4}
\highlightcellbyletter{2}{1}{2}
\addtext{0.8}{-1.6}{$P_{c(\pi^t)^{rc}}$}
\putright
\addcellrows {3+1}
\highlightcellbyletter{1}{1}{1}
\highlightcellbyletter{1}{2}{3}
\highlightcellbyletter{1}{3}{4}
\highlightcellbyletter{2}{1}{2}
\addtext{0.8}{-1.6}{$Q_{c(\pi^t)^{rc}}$}
\end{ferrers}
Then we apply inverse RSK on the pair $(P_{c(\pi^t)}, Q_{c(\pi^t)^{rc}})$ to get $\jmath(c(\pi^t))=4~1~2~3$, hence 
$\phi(\pi)^t=c^{-1}\circ \jmath \circ c(\pi^t)=9~6~7~8$. Similarly for $\phi(\pi)^b$, we apply RSK on $\pi^b=4~3~1~2$ and $(\pi^b)^{rc}=3~4~2~1$ to get the following four SYT.
\begin{ferrers}
\addcellrows {2+1+1}
\highlightcellbyletter{1}{1}{1}
\highlightcellbyletter{1}{2}{2}
\highlightcellbyletter{2}{1}{3}
\highlightcellbyletter{3}{1}{4}
\addtext{0.6}{-2}{$P_{\pi^b}$}
\putright
\addcellrows {2+1+1}
\highlightcellbyletter{1}{1}{1}
\highlightcellbyletter{1}{2}{4}
\highlightcellbyletter{2}{1}{2}
\highlightcellbyletter{3}{1}{3}
\addtext{0.6}{-2}{$Q_{\pi^b}$}
\putright
\addcellrows {2+1+1}
\highlightcellbyletter{1}{1}{1}
\highlightcellbyletter{1}{2}{4}
\highlightcellbyletter{2}{1}{2}
\highlightcellbyletter{3}{1}{3}
\addtext{0.6}{-2}{$P_{(\pi^b)^{rc}}$}
\putright
\addcellrows {2+1+1}
\highlightcellbyletter{1}{1}{1}
\highlightcellbyletter{1}{2}{2}
\highlightcellbyletter{2}{1}{3}
\highlightcellbyletter{3}{1}{4}
\addtext{0.6}{-2}{$Q_{(\pi^b)^{rc}}$}
\end{ferrers}
Reversing RSK on $(P_{\pi^b}, Q_{(\pi^b)^{rc}})$ then produces $\phi(\pi)^b=\jmath(\pi^b)=1~4~3~2$. Together we have
$$
(\phi(\pi)^t,\phi(\pi)^b,Sh(\phi(\pi)))=(9~6~7~8,\; 1~4~3~2,\; \{1,2,4,6,8\}),
$$
which corresponds uniquely to $\phi(\pi)=5~1~9~6~4~3~7~8~2\in \SS_{R(w)}$. Therefore $\phi_{R(w)}(v)=c^{-1}(\phi(\pi))=4~1~6~4~3~2~4~5~1$. Finally we check the statistics match up indeed (with $\MAJ$ and $\STAT$ switched).
$$
(\des, Id, \F, \STAT, \MAJ)\phi_{R(w)}(v)=(5, \{2,3,4,8\}, 4, 25, 21).
$$
\end{example}

\section{Further questions}\label{sec:conclusion}
Seeing our equidistribution result \textcolor{red}{in} Corollary~\ref{main-Rw}, a natural next step is to consider the so-called ``$k$-extension'' (see \cite{Han1, Han2, CF}). Namely, for our prototype of the alphabet, $[m]=\{1,2,\ldots,m\}$, we take two non-negative integers $l$ and $k$ such that $l+k=m$. Then we call the letters $1,2,\ldots,l$ \emph{small} and the letters $l+1,l+2,\ldots,l+k=m$ \emph{large}. Now consider any total ordering, say $\O$, on $[m]$ that is compatible with $k$, in the sense that for any small letter $x$ and any large letter $y$, we have $(x,y)\in \O$. Both $\des$ and $\MAJ$ have their $k$-extensions defined with respect to $\O$, so one may ask if there exists a $k$-extension for $\STAT$, such that the joint distribution of $(\des, \MAJ, \STAT)$ are still the same as that of $(\des, \STAT, \MAJ)$.
 
Finally, Corollary~\ref{main-Rw} justifies $(\des, \STAT)$ as a new Euler-Mahonian (over rearrangement class of words) pair in the family consisting of $(\des, \MAJ), (\des, \MAK)$ and $(\exc, \DEN)$ (see \cite{CSZ} for undefined statistics and further details). It would be interesting to pair other Mahonian statistics discovered in \cite{BS} with either $\des$ or $\exc$, and then pigeonhole the resulting pairs into the existing (or even new) families of Euler-Mahonian statistics. Another possible direction for furthering the study on $\STAT$ is to consider the equidistribution problems on pattern avoiding permutations, see the recent work in \cite{Am, Ch}.


\section*{Acknowledgement}
The first two authors were supported by the Fundamental Research Funds for the Central Universities (No.~CQDXWL-2014-Z004) and the National Natural Science Foundation of China (No.~11501061).

\end{document}